\documentclass{amsart}%
\usepackage{amssymb}
\usepackage{amsmath}
\usepackage{amsfonts}
\usepackage[all,cmtip]{xy}
\usepackage{hyperref}
\usepackage[margin=1.5in]{geometry}
\usepackage{diagxy}%
\usepackage{graphicx}

\newtheorem{theorem}{Theorem}
\theoremstyle{plain}
\newtheorem*{acknowledgement}{Acknowledgement}

\newtheorem{corollary}[theorem]{Corollary}

\newtheorem{definition}[theorem]{Definition}

\newtheorem{lemma}[theorem]{Lemma}

\newtheorem{proposition}[theorem]{Proposition}
\theoremstyle{remark}
\newtheorem{remark}{Remark}

\numberwithin{equation}{section}
\begin{document}
\title[Homology of Lie algebras like $\mathfrak{gl}(\infty,R)$]{On the homology of Lie algebras like $\mathfrak{gl}(\infty,R)$}
\author{Oliver Braunling}
\address{Freiburg Institute for Advanced Studies (FRIAS), University of Freiburg,
D-79104 Freiburg im\ Breisgau, Germany}
\thanks{The author was supported by DFG GK1821 \textquotedblleft Cohomological Methods
in Geometry\text{\textquotedblright}\ and a Junior Fellowship at the Freiburg
Institute for Advanced Studies (FRIAS)}

\begin{abstract}
We revisit a recent paper of Fialowski and Iohara. They compute the homology
of the Lie algebra $\mathfrak{gl}(\infty,R)$ for $R$ an associative unital
algebra over a field of characteristic zero. We explain how to obtain
essentially the same results by a completely different method.

\end{abstract}
\maketitle

This note is inspired by the recent paper \cite{fioh} of Fialowski and Iohara.
They compute the Lie algebra homology of the infinite matrix algebra
$\mathfrak{gl}(\infty,R)$. We shall recall below how this Lie algebra is
defined. Their paper takes up a thread of research initiated by Feigin
and\ Tsygan, and generalizes one of the main results of \cite{MR705056}.

While $\mathfrak{gl}(\infty,R)$ naturally acts on Laurent polynomials
$R[t,t^{-1}]$, in this note we consider a very closely related variant, which
acts in an analogous fashion on formal Laurent series $R((t))$. This leads to
a\textit{\ topologically} \textit{completed} variant of the same Lie algebra,
call it $\mathfrak{gl}_{\operatorname*{top}}(\infty,R)$. However, this little
change of perspective is only made for convenience, it is really not the main
point of this text. Instead, our focus is on computing the homology of
$\mathfrak{gl}_{\operatorname*{top}}(\infty,R)$ in a completely different
fashion than the methods used by Fialowski and Iohara.\medskip

Nonetheless, a posteriori it will turn out that $\mathfrak{gl}(\infty,R)$ and
$\mathfrak{gl}_{\operatorname*{top}}(\infty,R)$ have the same homology.

\section{\label{sect_1}Statement of the results}

Suppose $k$ is a field of characteristic zero and $R$ a unital associative $k
$-algebra. Define the \emph{algebra of generalized Jacobi matrices} (following
\cite{MR705056})%
\[
J(R):=\left\{  (a_{i,j})_{i,j\in\mathbb{Z}}\text{ with }a_{i,j}\in
R\mid\exists N:a_{i,j}=0\text{ for all }i,j\text{ with }\left\vert
i-j\right\vert >N\right\}  \text{.}%
\]
This is again a unital associative $k$-algebra if we equip it with the usual
matrix multiplication. The finiteness condition on the width of the diagonal
support ensures that this multiplication is well-defined.

Write $\mathfrak{gl}(\infty,R)$ to denote the Lie algebra of $J(R)$. Moreover,
let $\mathfrak{gl}_{n}(R)$ denote the usual Lie algebra of $(n\times
n)$-matrices with entries in $R$, and $\mathfrak{gl}_{\infty}(R):=\underset
{n\longrightarrow\infty}{\operatorname*{colim}}\mathfrak{gl}_{n}(R)$, where
$\mathfrak{gl}_{n}(R)$ is embedded into $\mathfrak{gl}_{n+1}(R)$ as the
upper-left $(n\times n)$-minor, and the remaining $(n+1)$-st row and column
being set to zero. Again, $\mathfrak{gl}_{\infty}(R)$ is a Lie algebra.
Despite the similar notation, it is very different from $\mathfrak{gl}%
(\infty,R)$.

The main results of \cite{fioh} are:

\begin{theorem}
[Fialowski--Iohara \cite{fioh}]\label{thm_FI}Let $R$ be an associative unital
$k$-algebra, where $k$ is a field of characteristic zero.

\begin{enumerate}
\item There is a canonical isomorphism of graded algebras $HC_{\bullet
}(J(R))\cong HC_{\bullet-1}(R)$.

\item There is a canonical isomorphism of graded algebras $HH_{\bullet
}(J(R))\cong HH_{\bullet-1}(R)$.

\item There is a canonical isomorphism of the primitive parts
$\operatorname*{Prim}H_{\bullet}(\mathfrak{gl}(\infty,R),k)\cong%
\operatorname*{Prim}H_{\bullet}(\mathfrak{gl}_{\infty}(R),k)[-1]$.
\end{enumerate}
\end{theorem}

By speaking of the primitive part, we refer to the fact that the Lie homology
algebra $H_{\bullet}(\mathfrak{g},k)$, for an arbitrary Lie algebra
$\mathfrak{g}$ over a field $k$, has a graded cocommutative coalgebra
structure with counit, induced from the comultiplication%
\[
\triangle:H_{\bullet}(\mathfrak{g},k)\overset{\delta}{\longrightarrow
}H_{\bullet}(\mathfrak{g}\oplus\mathfrak{g},k)\overset{\sigma}{\longrightarrow
}H_{\bullet}(\mathfrak{g},k)\otimes_{k}H_{\bullet}(\mathfrak{g},k)\text{,}%
\]
where $\delta:\mathfrak{g}\rightarrow\mathfrak{g}\oplus\mathfrak{g}$ is the
diagonal $g\mapsto(g,g)$, and $\sigma$ denotes the K\"{u}nneth isomorphism for
Lie homology. Thus, it has a notion of primitive elements, namely those $x$
satisfying $\triangle x=1\otimes x+x\otimes1$.

We define a variation of $J(R)$:\ We may take%
\[
E(R):=\left\{  \varphi\in\operatorname*{End}\nolimits_{R}(\left.
R((t))\right.  )\left\vert
\begin{array}
[c]{c}%
\text{(1) for every }n\in\mathbb{Z}\text{ there exists some }n^{\prime}%
\in\mathbb{Z}\\
\text{such that }\varphi(t^{n}R[[t]])\subseteq t^{n^{\prime}}R[[t]]\text{,
\textit{and}}\\
\text{(2) for every }m\in\mathbb{Z}\text{ there exists some }m^{\prime}%
\in\mathbb{Z}\\
\text{such that }\varphi(t^{m^{\prime}}R[[t]])\subseteq t^{m}R[[t]]
\end{array}
\right.  \right\}
\]
and let $\mathfrak{gl}_{\operatorname*{top}}(\infty,R)$ be its Lie algebra.

\begin{theorem}
(Theorem \ref{thm_body_1}) Let $R$ be an associative unital $k$-algebra, where
$k$ is a field of characteristic zero.

\begin{enumerate}
\item There is a canonical isomorphism of graded algebras $HC_{\bullet
}(E(R))\cong HC_{\bullet-1}(R)$,

\item There is a canonical isomorphism of graded algebras $HH_{\bullet
}(E(R))\cong HH_{\bullet-1}(R)$,

\item There is a canonical isomorphism of the primitive parts
$\operatorname*{Prim}H_{\bullet}(\mathfrak{gl}_{\operatorname*{top}}%
(\infty,R),k)\cong\operatorname*{Prim}H_{\bullet}(\mathfrak{gl}_{\infty
}(R),k)[-1]$.
\end{enumerate}
\end{theorem}

Of course this statement has exactly the same shape as the theorem of
Fialowski--Iohara. But even though this is a statement entirely about
associative algebras and Lie algebras, we shall prove it using a detour
through some \textit{category-theoretic concepts}. The key input for the proof
will be a localization theorem for categories. No spectral sequences or
auxiliary homology computations as in \cite{fioh} will appear. So, the method
of proof is quite different. We will explain below how $E$ and $\mathfrak{gl}%
_{top}(\infty,R)$ are just topologically completed variants of $J$ and
$\mathfrak{gl}(\infty,R)$; and note that that since the right-hand side terms
in our theorem are precisely those as in the one of Fialowski--Iohara, this
topological completion obviously does not affect the homology at all.

Mild variations of the same proof also give a variant for group homology.
Write $\operatorname{GL}_{\infty}(R)$ for $\underset{n\longrightarrow\infty
}{\operatorname*{colim}}\operatorname{GL}_{n}(R)$, where $\operatorname{GL}%
_{n}(R)$ is embedded into $\operatorname{GL}_{n+1}(R)$ as the upper-left
$(n\times n)$-minor, and the remaining $(n+1)$-st row and column being set to
zero, except for the $(n+1,n+1)$ entry, which is set to the identity.

Although there does not exist an equally close connection between
$\operatorname{GL}_{\infty}(R)$ and $\mathfrak{gl}_{\infty}(R)$ as in the
world of real Lie groups and their Lie algebras, the analogous group homology
fact still holds true.

\begin{theorem}
\label{thm_intro_2}(Theorem \ref{thm_body_2}) Let $R$ be an associative unital
$k$-algebra. There is a canonical isomorphism of primitive parts
$\operatorname*{Prim}H_{n}(\operatorname{GL}_{\infty}(E(R)),\mathbb{Q}%
)\cong\operatorname*{Prim}H_{n-1}(\operatorname{GL}_{\infty}(R),\mathbb{Q})$
for all $n\geq2$.
\end{theorem}

We believe that the same statement is probably also true for $J(R)$ instead of
$E(R)$, but we have no proof.

\section{\label{sect_2}The Lie algebras $\mathfrak{gl}(\infty,R)$ and
$\mathfrak{gl}_{\operatorname*{top}}(\infty,R)$}

We have recalled Feigin and Tsygan's Lie algebra $\mathfrak{gl}(\infty,R)$ in
the previous section. Now let us define $\mathfrak{gl}_{\operatorname*{top}%
}(\infty,R)$. We shall use a slightly more complicated approach than we have
used in Section \ref{sect_1} (we shall prove in Prop. \ref{prop_E} that the
definition in Section \ref{sect_1} is equivalent). Even though we only need a
Lie algebra, let us take a quick detour through a category-theoretic approach:
For every exact category $\mathcal{C}$, there is a commutative diagram of
exact functors%
\begin{equation}
\bfig\Square(0,0)[\mathcal{C}`\mathsf{Ind}^{a}_{\aleph_{0}}\mathcal{C}%
`\mathsf{Pro}^{a}_{\aleph_{0}}\mathcal{C}`\mathsf{Tate}^{el}_{\aleph_{0}%
}\mathcal{C}\text{,};```]
\efig
\label{lab8}%
\end{equation}
where $\mathsf{Ind}_{\aleph_{0}}^{a}(\mathcal{C})$, $\mathsf{Pro}_{\aleph_{0}%
}^{a}(\mathcal{C})$ refer to certain variations of the classical Ind- and
Pro-categories (the main difference is that we only allow defining
Ind-diagrams whose transition morphisms are admissible monics, resp.
Pro-diagrams whose transition morphisms are admissible epics with respect to
the exact structure of $\mathcal{C}$). The category $\mathsf{Tate}_{\aleph
_{0}}^{el}(\mathcal{C})$ can be defined as the full subcategory of
$\mathsf{Ind}_{\aleph_{0}}^{a}(\mathsf{Pro}_{\aleph_{0}}^{a}(\mathcal{C}))$
consisting of those objects which admit a Pro-subobject with an Ind-object
quotient, i.e. which can be presented as%
\[
L\hookrightarrow X\twoheadrightarrow X/L\text{,}%
\]
where $L\in\mathsf{Pro}_{\aleph_{0}}^{a}(\mathcal{C})$ and $X/L\in
\mathsf{Ind}_{\aleph_{0}}^{a}(\mathcal{C})$; such a subobject is known as a
\emph{lattice}. Here we understand this Ind- and Pro-category as canonically
embedded into $\mathsf{Ind}_{\aleph_{0}}^{a}(\mathsf{Pro}_{\aleph_{0}}%
^{a}(\mathcal{C}))$ by writing the Ind-diagram as an Ind-Pro diagram which is
constant in the Pro-direction, resp. a Pro-diagram as an Ind-Pro diagram which
is constant in the Ind-direction. We refer to \cite{MR2872533} or
\cite{MR3510209} for precise definitions. If $\mathcal{C}$ is a $k$-linear
exact category, then so is $\mathsf{Tate}_{\aleph_{0}}^{el}(\mathcal{C})$, in
a canonical fashion.

If $\mathcal{C}=\mathsf{P}_{f}(R)$ is the $k$-linear exact category of
finitely generated projective $R$-modules, there is a special object in
$\mathsf{Tate}_{\aleph_{0}}^{el}(\mathcal{C})$:
\begin{equation}
``R((t))\text{\textquotedblright}:=\coprod_{n\in\mathbb{N}}R\oplus\prod
_{n\in\mathbb{N}}R\label{lab10}%
\end{equation}
or equivalently (isomorphically)%
\begin{equation}
``R((t))\text{\textquotedblright}\cong\underset{n}{\underrightarrow
{\operatorname*{colim}}}\underset{m}{\underleftarrow{\lim}}\left.  \frac
{1}{t^{n}}R[t]\right/  t^{m}R[t]\text{,}\label{lab10d}%
\end{equation}
which might be a more suggestive way of defining an object which we can think
of as a formal incarnation of Laurent series. The former presentation makes it
clear that $``R((t))\text{\textquotedblright}$ lies in $\mathsf{Ind}_{\aleph_{0}}%
^{a}(\mathsf{Pro}_{\aleph_{0}}^{a}(\mathcal{C}))$, and since $\prod
_{n\in\mathbb{N}}R$ is visibly a lattice, $``R((t))\text{\textquotedblright}$ lies in
$\mathsf{Tate}_{\aleph_{0}}^{el}(\mathcal{C})$.

Despite the fact that these category-theoretic constructions might appear
quite fancy, perhaps even too fancy, we only need them in a very special
situation where all the categories boil down to be equivalent to categories of
projective modules:

Let $\mathsf{Tate}_{\aleph_{0}}(\mathcal{C})$ denote the idempotent completion
of the exact category $\mathsf{Tate}_{\aleph_{0}}^{el}(\mathcal{C})$.

\begin{lemma}
[Key Lemma]\label{lemma_ModuleCat}If $\mathcal{C}=\mathsf{P}_{f}(R)$ is the
$k$-linear exact category of finitely generated projective $R$-modules, then
there is a $k$-linear exact equivalence of $k$-linear exact categories
$\mathsf{Tate}_{\aleph_{0}}(\mathsf{P}_{f}(R))\cong\mathsf{P}_{f}(E(R))$,
where%
\begin{equation}
E(R):=\operatorname*{End}\nolimits_{\mathsf{Tate}_{\aleph_{0}}(\mathsf{P}%
_{f}(R))}(``R((t))\text{\textquotedblright})\text{.}\label{lab1}%
\end{equation}
Indeed, $``R((t))\text{\textquotedblright}$ is a projective generator of
$\mathsf{Tate}_{\aleph_{0}}(\mathsf{P}_{f}(R))$.
\end{lemma}

\begin{proof}
This is \cite{MR3621099}, Theorem 1, applied for $\mathcal{C}=\mathsf{P}%
_{f}(R)$ and $n=1$. Note that the countable collection $\{S_{i}\}$ of the
theorem loc. cit. can be taken to be the single object $R$ (as a module). The
reference does not discuss the additional $k$-linear structure, but an
inspection of the proof loc. cit. easily shows that this structure is also
compatible under the given equivalence.
\end{proof}

We take the Lie algebra of the associative algebra underlying this projective
module category as our definition for a topologically completed analogue of
$\mathfrak{gl}(\infty,R)$:

\begin{definition}
Let $k$ be a field and $R$ a unital associative $k$-algebra. Define
$\mathfrak{gl}_{\operatorname*{top}}(\infty,R)$ as the Lie algebra of the
unital associative $k$-algebra $E(R)$.
\end{definition}

We need to justify how $\mathfrak{gl}_{\operatorname*{top}}(\infty,R)$
resembles the infinite matrix algebra $\mathfrak{gl}(\infty,R)$ of Fialowski
and Iohara.

\begin{remark}
Let us dwell on a bit more about the differences of $R[t,t^{-1}]$ (the
situation of Feigin--Tsygan \cite{MR705056} and Fialowski--Iohara
\cite{fioh}), in comparison to $R((t))$. For simplicitly, let us focus on
$R:=k$. Then the big advantage of $k[t,t^{-1}]$ is that we can specify an
explicit $k$-vector space basis. This is not possible for $k((t))$. However,
$k((t))$ has the advantage to be closed under dualization; noting that
$k((t))\simeq tk[[t]]\oplus k[t^{-1}]$ as a $k$-vector space, and the two
direct summands are dual to each other under taking the continuous $k$-dual,
and equipping $k[[t]]$ with the natural $t$-adic linear topology, and
$k[t^{-1}]$ with the discrete topology.
\end{remark}

Let us give a more explicit description of the $-$ quite abstract $-$
definition in Equation \ref{lab1}. After all, we prefer to be able to specify
$E(R)$ without having to rely on a category-theoretic picture:

\begin{proposition}
\label{prop_E}There is a canonical isomorphism of unital associative algebras,%
\begin{equation}
E(R)\overset{\sim}{\longrightarrow}\left\{  \varphi\in\operatorname*{End}%
\nolimits_{R}(\left.  R((t))\right.  )\left\vert
\begin{array}
[c]{c}%
\text{(1) for every }n\in\mathbb{Z}\text{ there exists some }n^{\prime}%
\in\mathbb{Z}\\
\text{such that }\varphi(t^{n}R[[t]])\subseteq t^{n^{\prime}}R[[t]]\text{,
\textit{and}}\\
\text{(2) for every }m\in\mathbb{Z}\text{ there exists some }m^{\prime}%
\in\mathbb{Z}\\
\text{such that }\varphi(t^{m^{\prime}}R[[t]])\subseteq t^{m}R[[t]]
\end{array}
\right.  \right\}  \text{.}\label{lab3a}%
\end{equation}
In particular, $\mathfrak{gl}_{\operatorname*{top}}(\infty,R)$ is the\ Lie
algebra of the right-hand side subalgebra of $\operatorname*{End}%
\nolimits_{R}(\left.  R((t))\right.  )$ under the commutator bracket.
\end{proposition}

\begin{proof}
\textit{(Step 1)} First of all, define an exact functor%
\[
G:\mathsf{Tate}_{\aleph_{0}}(\mathsf{P}_{f}(R))\longrightarrow\mathsf{Mod}%
(R)\text{,}\qquad\left.  ``R((t))\text{\textquotedblright}\right.  \longmapsto
R((t))\text{.}%
\]
As $``R((t))\text{\textquotedblright}$ is a projective generator, for defining an
exact functor it suffices to determine the object this generator is sent to,
and moreover make sure that every endomorphism of the projective generator
also defines an endomorphism in $\mathsf{Mod}(R)$. To this end, we unravel the
definition of morphisms in the category on the left:\ Recall that
$\mathsf{Tate}_{\aleph_{0}}(\mathsf{P}_{f}(R))$ is the idempotent completion
of $\mathsf{Tate}_{\aleph_{0}}^{el}(\mathsf{P}_{f}(R))$, and $\left.
``R((t))\text{\textquotedblright}\right.  $ lies in the latter, so it suffices to
understand the Hom-sets in $\mathsf{Tate}_{\aleph_{0}}^{el}(\mathsf{P}%
_{f}(R))$ itself. By definition,
\[
\mathsf{Tate}_{\aleph_{0}}^{el}(\mathsf{P}_{f}(R))\subset\mathsf{Ind}%
_{\aleph_{0}}^{a}\mathsf{Pro}_{\aleph_{0}}^{a}(\mathsf{P}_{f}(R))
\]
is a full subcategory, so this reduces to evaluating the Hom-sets of Ind-
resp. Pro-objects individually. For admissible Ind-objects of an arbitrary
exact category $\mathcal{C}$ we have
\[
\operatorname*{Hom}\nolimits_{\mathsf{Ind}_{\aleph_{0}}^{a}(\mathcal{C}%
)}(X_{\bullet},Y_{\bullet})=\underset{i}{\underleftarrow{\lim}}\underset
{j}{\underrightarrow{\operatorname*{colim}}}\operatorname*{Hom}%
\nolimits_{\mathcal{C}}(X_{i},X_{j})\text{,}%
\]
where $X_{\bullet}:I\rightarrow\mathcal{C}$ is an admissible Ind-diagram
(analogously for $Y_{\bullet}$) representing the object, with $I$ the
countable index category. Since $\mathsf{Pro}^{a}(\mathcal{C})=\mathsf{Ind}%
^{a}(\mathcal{C}^{op})^{op}$, this also describes the Hom-sets of Pro-objects.
In total, we find%
\begin{equation}
\operatorname*{Hom}\nolimits_{\mathsf{Tate}_{\aleph_{0}}^{el}(\mathsf{P}%
_{f}(R))}(X_{\bullet,\bullet},Y_{\bullet,\bullet})=\underset{I}%
{\underleftarrow{\lim}}\underset{J}{\underrightarrow{\operatorname*{colim}}%
}\underset{j}{\underleftarrow{\lim}}\underset{i}{\underrightarrow
{\operatorname*{colim}}}\operatorname*{Hom}\nolimits_{R}(X_{I,i}%
,X_{J,j})\text{.}\label{le1}%
\end{equation}
If we write%
\[
R((t))=\underset{n}{\underrightarrow{\operatorname*{colim}}}\underset
{m}{\underleftarrow{\lim}}\left.  \frac{1}{t^{n}}R[t]\right/  t^{m}R[t]
\]
in $\mathsf{Mod}(R)$, but not as an Ind-Pro object, but as a concrete colimit
and limit in the bi-complete category $\mathsf{Mod}(R)$, we see that an
element of Equation \ref{le1} for the Tate object
$X=Y=``R((t))\text{\textquotedblright}$ indeed defines an $R$-module endomorphism of
$R((t))$. Moreover, the colimits in Equation \ref{le1} are along (split)
injections, and the limits along split surjections, so the map%
\[
\operatorname*{Hom}\nolimits_{\mathsf{Tate}_{\aleph_{0}}^{el}(\mathsf{P}%
_{f}(R))}(X,Y)\longrightarrow\operatorname*{Hom}\nolimits_{R}(GX,GY)
\]
is not just well-defined, it is also injective. Thus, $G$ is an exact and
faithful functor. \textit{(Step 2)}\ The faithfulness implies that we have an
inclusion of unital associative algebras $E(R)\subseteq\operatorname*{End}%
\nolimits_{R}(\left.  R((t))\right.  )$. It remains to show that the image
agrees with the description which we give in our claim. We note that for all
$n\in\mathbb{Z}$ the subobject $``t^{n}\cdot R[[t]]\text{\textquotedblright}$, which
is a Pro-object, is indeed a lattice in $``R((t))\text{\textquotedblright}$, because%
\[
``t^{n}\cdot R[[t]]\text{\textquotedblright}\hookrightarrow``R((t))\text{\textquotedblright}%
\twoheadrightarrow``R((t))\left/  t^{n}\cdot R[[t]]\right.  \text{\textquotedblright}%
\]
and the latter is isomorphic to $``t^{n-1}\cdot R[t^{-1}]\text{\textquotedblright}$,
which is an Ind-object in turn. So all $``t^{n}\cdot R[[t]]\text{\textquotedblright}$
are lattices, and they are bi-final in the poset of lattices, meaning that
every lattice $L$ admits some $n,n^{\prime}\in\mathbb{Z}$ with%
\begin{equation}
``t^{n}\cdot R[[t]]\text{\textquotedblright}\hookrightarrow L\hookrightarrow
``t^{n^{\prime}}\cdot R[[t]]\text{\textquotedblright}\text{.}\label{lab2a}%
\end{equation}
Now suppose $\varphi\in E(R)=\operatorname*{End}\nolimits_{\mathsf{Tate}%
_{\aleph_{0}}(\mathsf{P}_{f}(R))}(``R((t))\text{\textquotedblright})$. Since
$``t^{n}\cdot R[[t]]\text{\textquotedblright}$ is a lattice in
$``R((t))\text{\textquotedblright}$, \cite[Lemma 1, (1)]{MR3621099} implies that
there exists a lattice $L$ in $``R((t))\text{\textquotedblright}$ such that
$\varphi(``t^{n}\cdot R[[t]]\text{\textquotedblright})\subseteq L$ and since the
lattices of the shape $``t^{\mathbb{Z}}\cdot R[[t]]\text{\textquotedblright}$ are
co-final among all lattices, this $L$ is contained in some $``t^{n^{\prime}%
}\cdot R[[t]]\text{\textquotedblright}$ for a suitable $n^{\prime}$, as in Equation
\ref{lab2a}. It follows that Condition (1) in Equation \ref{lab3a} is met. For
Condition (2) argue analogously, this time using \cite[Lemma 1, (2)]%
{MR3621099} instead. This proves that under the inclusion $E(R)\subseteq
\operatorname*{End}\nolimits_{R}(\left.  R((t))\right.  )$ the elements
$\varphi\in E(R)$ indeed meet the Conditions (1), (2) of our claim. It remains
to show that whenever we are given a $\varphi\in\operatorname*{End}%
\nolimits_{R}(\left.  R((t))\right.  )$ meeting Conditions (1), (2), that this
defines an endomorphism in $\operatorname*{End}\nolimits_{\mathsf{Tate}%
_{\aleph_{0}}(\mathsf{P}_{f}(R))}(``R((t))\text{\textquotedblright})$. But this is
harmless: Observing\ Equation \ref{lab10d}, Conditions (1), (2) just mean
that, up to re-indexing $n,m$, $\varphi$ induces a (possibly non-straight)
morphism in the underlying\ Ind- and Pro-category. This proves that Conditions
(1), (2) precisely determine the image of $E(R)$ inside $\operatorname*{End}%
\nolimits_{R}(\left.  R((t))\right.  )$. Thus, for the isomorphism of unital
associative algebras which we claim to exist in Equation \ref{lab3a}, we may
just take the inclusion map. We already know that this is an algebra
homomorphism, Step 1 shows that it is injective, and Step 2 that it is
surjective. The claim follows.
\end{proof}

We may also compare $J(R)$ and $E(R)$ on a more structural level:

Firstly, the $J(R)$ of \cite{fioh} comes equipped with two-sided ideals%
\begin{align}
I^{+}(R)  & :=\{(a_{i,j})\in J(R)\mid\exists B_{a}:i<B_{a}\Rightarrow
a_{i,j}=0\}\label{lab2}\\
I^{-}(R)  & :=\{(a_{i,j})\in J(R)\mid\exists B_{a}:j>B_{a}\Rightarrow
a_{i,j}=0\}\text{.}\nonumber
\end{align}
These have the following properties:%
\[
I^{0}(R):=I^{+}(R)\cap I^{-}(R)\qquad\text{and}\qquad I^{+}(R)+I^{-}%
(R)=J(R)\text{.}%
\]
See \cite[\S 1.1, loc. cit. we write $E(R)$ instead of $J(R)$, but for the
present text we prefer to follow Feigin--Tsygan's notation.]{MR3207578} for
more on this. The ideal $I^{0}(R)$ consists precisely of those matrices with
only finitely many non-zero entries, so the usual definition of a matrix trace
makes sense on it; we could call $I^{0}(R)$ the trace-class elements. In fact,
all this means that the generalized Jacobi matrices are an example of a
$1$-fold Beilinson cubical algebra in the sense of \cite[Definition
1]{MR3621099}.

The algebra $J(R)$ naturally acts on the Laurent polynomial ring $R[t,t^{-1}]
$. Write polynomials as $f=\sum_{i\in\mathbb{Z}}f_{i}t^{i}$, also denoted
$f=(f_{i})_{i}$ with $f_{i}\in R$, and let $a=(a_{i,j})$ act by $\left(
a\cdot f\right)  _{i}:=\sum_{k\in\mathbb{Z}}a_{i,k}f_{k}$. This makes $J(R)$ a
submodule of the right $R$-module endomorphisms $\operatorname*{End}%
_{R}(R[t,t^{-1}])$.

Now, let us explain how this differs from $E(R)$: Instead of Laurent
polynomials, we could also consider formal Laurent series, that is
$R((t)):=R[[t]][t^{-1}]$. Clearly, $R[t,t^{-1}]\subseteq R((t))$. As in
Equation \ref{lab1}, we define $E(R)$ to be the endomorphism algebra of this
object in this category. Then%
\begin{align*}
I^{+}(R)  & :=\{a\in\operatorname*{End}(``R((t))\text{\textquotedblright}%
)\mid\text{the image of }a\text{ is contained in a lattice}\}\text{,}\\
I^{-}(R)  & :=\{a\in\operatorname*{End}(``R((t))\text{\textquotedblright})\mid\text{a
lattice is mapped to zero under }a\}
\end{align*}
are two-sided ideals, now inside of $E(R)$ instead of $J(R)$, and again%
\[
I^{0}(R):=I^{+}(R)\cap I^{-}(R)\qquad\text{and}\qquad I^{+}(R)+I^{-}(R)=E(R)
\]
holds. We chose to pick the same notation for the ideals to stress the analogy
with Equations \ref{lab2}. In particular, we again obtain an example of a
$1$-fold Beilinson cubical algebra. All these properties are a special case of
\cite[Theorem 1]{MR3621099}, using that the projective modules $\mathsf{P}%
_{f}(R)$ form an idempotent complete split exact category.

\section{Proof}

Fix a field $k$ of arbitrary characteristic. Below, whenever we speak of
cyclic or Hochschild homology, we tacitly regard $k$ as the base field, even
if we do not repeat mentioning $k$ in the notation.

Cyclic homology is classically only defined for $k$-algebras, but thanks to
the work of Keller \cite{MR1667558}, we also have a concept of cyclic homology
for $k$-linear exact categories. In the case of $k$-algebras, we have
agreement in the sense that%
\[
HC(R)\cong HC(\mathsf{P}_{f}(R))\text{,}%
\]
where on the left-hand side $HC(R)$ denotes the ordinary definition of cyclic
homology, for example as in \cite{MR1217970}, whereas on the right-hand side
we use Keller's definition. For the proof of agreement see \cite[\S 1.5,
Theorem, (a)]{MR1667558}.

Let us briefly recall the\ Eilenberg swindle, in a quite general format, using
the language of \cite{MR3070515}:

\begin{lemma}
[Eilenberg swindle]\label{lem_EilenbergSwindle}Suppose an exact category
$\mathcal{C}$ is closed under countable products or coproducts and
$I:\operatorname*{Cat}_{\infty}^{\operatorname*{ex}}\rightarrow\mathsf{Sp}$ is
any additive invariant with values in spectra $\mathsf{Sp}$ (e.g., Hochschild
homology or cyclic homology), then $I(\mathcal{C})=0$.
\end{lemma}

We formulate this in terms of spectra here, but since we only care about
Hochschild or cyclic homology, simplicial abelian groups would suffice.

\begin{proof}
Write $\mathcal{EC}$ to denote the exact category of exact sequences in
$\mathcal{C}$. Suppose $\mathcal{C}$ has countable coproducts. Define%
\[
E:\mathcal{C}\longrightarrow\mathcal{EC}\text{,}\quad X\mapsto
(X\hookrightarrow{\coprod\nolimits_{i\in\mathbb{Z}}}X\twoheadrightarrow
{\coprod\nolimits_{i\in\mathbb{Z}}}X)\text{.}%
\]
The admissible epic on the right is the shift-one-to-the-left functor in the
indexing of the coproducts. We can project to the left, middle or right term
of $\mathcal{EC}$, write $p_{i}:\mathcal{C}\rightarrow\mathcal{C}$ with
$i=0,1,2$ for these functors. The additivity of the invariant $I$ implies that
$p_{1}=p_{0}+p_{2}$ holds for the maps induced on $I(\mathcal{C})\rightarrow
I(\mathcal{C})$. Since $p_{1}=p_{2}$, we must have $p_{0}=0$. Since $p_{0}$ is
induced from the identity functor, $p_{0}$ must also be the identity, and thus
$I(\mathcal{C})=0$. The argument in the case of $\mathcal{C}$ having countable
products is analogous.
\end{proof}

That cyclic homology and Hochschild homology are additive invariants in the
sense of the above lemma, was proven by Keller \cite[\S 1.1.2, Theorem]%
{MR1667558}.

If $\mathcal{C}$ is an idempotent complete exact category and $\mathcal{C}%
\hookrightarrow\mathcal{D}$ a right $s$-filtering inclusion as a full
subcategory of an exact category $\mathcal{D}$, then%
\begin{equation}
D^{b}(\mathcal{C})\hookrightarrow D^{b}(\mathcal{D})\twoheadrightarrow
D^{b}(\mathcal{D}/\mathcal{C})\label{loa1}%
\end{equation}
is an exact sequence of triangulated categories by work of Schlichting, see
\cite[Prop. 2.6]{MR2079996}. From this one can deduce the following variant of
a theorem of\ Keller:

\begin{theorem}
[Keller's Localization Theorem]\label{Thm_KellersLocalizationTheorem}Let
$\mathcal{C}$ be an exact category and $\mathcal{C}\hookrightarrow\mathcal{D}$
a right $s$-filtering inclusion as a full subcategory of an exact category
$\mathcal{D}$. Then%
\[
HC(\mathcal{C})\longrightarrow HC(\mathcal{D})\longrightarrow HC(\mathcal{D}%
/\mathcal{C})
\]
is a fiber sequence in cyclic homology. The analogous statement for Hochschild
homology is also true.
\end{theorem}

\begin{proof}
Let us provide some details to translate this formulation into the one of
Keller's paper. If $\mathcal{C}$ is idempotent complete, the claim literally
follows from Schlichting's result, Equation \ref{loa1}, and Keller's
\cite[\S 1.5, Theorem]{MR1667558}. If $\mathcal{C}$ is not idempotent
complete, we need to work a little harder, but the necessary tools are all
available thanks to \cite{MR2079996}: Instead of Equation \ref{loa1},
Schlichting also provides a slightly more general variant: If $\mathcal{C}$ is
not necessarily idempotent complete, he constructs an exact category
$\widetilde{\mathcal{D}}^{\mathcal{C}}$ (cf. \cite[Lemma 1.20]{MR2079996})
such that%
\[
\mathcal{D}\hookrightarrow\widetilde{\mathcal{D}}^{\mathcal{C}}\hookrightarrow
\mathcal{D}^{ic}\text{,}%
\]
with (1) $\widetilde{\mathcal{D}}^{\mathcal{C}}\hookrightarrow\mathcal{D}%
^{ic}$ an extension-closed full subcategory, and (2) $\mathcal{D}%
\hookrightarrow\widetilde{\mathcal{D}}^{\mathcal{C}}$ co-final (a.k.a.
`factor-dense' in Keller's language \cite[\S 1.5]{MR1667558}), (3) the right
$s$-filtering embedding $\mathcal{C}\hookrightarrow\mathcal{D}$, inducing
$\mathcal{C}^{ic}\hookrightarrow\mathcal{D}^{ic}$ by $2$-functoriality,
factors over a right $s$-filtering embedding $\mathcal{C}^{ic}\hookrightarrow
\widetilde{\mathcal{D}}^{\mathcal{C}}$ and finally, (4), there is an exact
equivalence of exact categories $\mathcal{D}/\mathcal{C}\overset{\sim
}{\longrightarrow}\widetilde{\mathcal{D}}^{\mathcal{C}}/\mathcal{C}^{ic}$. See
Schlichting's paper \cite[Lemma 1.20]{MR2079996}. We can use this tool now and
proceed as follows: From the assumptions of the theorem as our input, we begin
with Schlichting's construct $\mathcal{C}^{ic}\hookrightarrow\widetilde
{\mathcal{D}}^{\mathcal{C}}\twoheadrightarrow\widetilde{\mathcal{D}%
}^{\mathcal{C}}/\mathcal{C}^{ic}$, so that Equation \ref{loa1} yields the
exact sequence of triangulated categories $D^{b}(\mathcal{C}^{ic}%
)\hookrightarrow D^{b}(\widetilde{\mathcal{D}}^{\mathcal{C}}%
)\twoheadrightarrow D^{b}(\widetilde{\mathcal{D}}^{\mathcal{C}}/\mathcal{C}%
^{ic})$. Keller's original formulation in \cite{MR1667558} takes such a
sequence as its input, giving the fiber sequence%
\[
HC(\mathcal{C}^{ic})\longrightarrow HC(\widetilde{\mathcal{D}}^{\mathcal{C}%
})\longrightarrow HC(\widetilde{\mathcal{D}}^{\mathcal{C}}/\mathcal{C}%
^{ic})\text{.}%
\]
But Keller's cyclic homology is invariant up to exact equivalence up to
factors/co-finally (see \cite[\S 1.5, Theorem, (b)]{MR1667558}), so%
\[
HC(\widetilde{\mathcal{D}}^{\mathcal{C}})\cong HC(\mathcal{D})\qquad
\text{and}\qquad HC(\widetilde{\mathcal{D}}^{\mathcal{C}}/\mathcal{C}%
^{ic})\cong HC(\mathcal{D}/\mathcal{C})\text{.}%
\]
Finally, $\mathcal{C}\hookrightarrow\mathcal{C}^{ic}$ is co-final, and the
idempotent completion of an exact category commutes with the idempotent
completion of its bounded derived category, $D^{b}(\mathcal{C}^{ic}%
)=D^{b}(\mathcal{C})^{ic}$, by Balmer and\ Schlichting \cite[Corollary
2.12]{MR1813503}, so that $D^{b}(\mathcal{C})\rightarrow D^{b}(\mathcal{C}%
^{ic})$ is also co-final and thus an equivalence up to factors itself.
\end{proof}

Based on the Eilenberg swindle, we obtain the fundamental delooping property
of Tate categories. This transports an idea of Sho Saito from algebraic
$K$-theory into the setup of the present note:

\begin{theorem}
\label{thm_deloop}Let $k$ be any field. Suppose $\mathcal{C}$ is an idempotent
complete $k$-linear exact category. Then there are canonical isomorphisms%
\[
HC_{\bullet}(\mathsf{Tate}_{\aleph_{0}}(\mathcal{C}))\cong HC_{\bullet
-1}(\mathcal{C})\qquad\text{and}\qquad HH_{\bullet}(\mathsf{Tate}_{\aleph_{0}%
}(\mathcal{C}))\cong HH_{\bullet-1}(\mathcal{C})\text{.}%
\]

\end{theorem}

\begin{proof}
The delooping property of (countably indexed) Tate categories \textit{for
algebraic }$K$\textit{-theory} was first conjectured by Kapranov in an
unpublished letter to Brylinski, later first steps to a proof were taken in
Previdi's thesis, and eventually proven in full generality by Sho Saito
\cite{MR3317759}. By inspecting Sho Saito's proof, one quickly realizes that
it can be adapted to prove the same delooping property for cyclic homology and
Hochschild homology. Let us explain this: Using that $\mathcal{C}%
\hookrightarrow\mathsf{Pro}_{\aleph_{0}}^{a}(\mathcal{C})$ is right
$s$-filtering (see loc. cit., or \cite[Theorem 4.2]{MR3510209}), and that
$\mathsf{Ind}_{\aleph_{0}}^{a}(\mathcal{C})\hookrightarrow\mathsf{Tate}%
_{\aleph_{0}}^{el}(\mathcal{C})$ is right $s$-filtering (loc. cit., or
\cite[Remark 5.3.5]{MR3510209}), we get the homotopy commutative diagram%
\begin{equation}%
\xymatrix{ {HC(\mathcal{C})} \ar[rr] \ar[d] && {HC(\mathsf{Pro}_{\aleph_{0}%
}^{a}(\mathcal{C}))} \ar[rr] \ar[d] && {HC(\mathsf{Pro}_{\aleph_{0}}%
^{a}(\mathcal{C})/\mathcal{C})} \ar[d] \\
{HC(\mathsf{Ind}_{\aleph_{0}}^{a}(\mathcal{C}))} \ar[rr] && {HC(\mathsf
{Tate}_{\aleph_{0}}^{el}(\mathcal{C}))} \ar[rr] && {HC(\mathsf{Tate}%
_{\aleph_{0}}^{el}(\mathcal{C})/\mathsf{Ind}_{\aleph_{0}}^{a}(\mathcal{C}))}
}%
\label{lmai_30}%
\end{equation}
where both rows are fiber sequences. This follows from Keller's Localization
Theorem,\ in the special formulation we have given above as Theorem
\ref{Thm_KellersLocalizationTheorem}. The downward arrows are induced from the
natural exact functors between the respective categories. Adapting Saito's
idea, the right-hand side downward map stems from an exact equivalence,%
\[
\left.  \mathsf{Tate}_{\aleph_{0}}^{el}(\mathcal{C})/\mathsf{Ind}_{\aleph_{0}%
}^{a}(\mathcal{C})\right.  \overset{\sim}{\longrightarrow}\left.
\mathsf{Pro}_{\aleph_{0}}^{a}(\mathcal{C})/\mathcal{C}\right.
\]
(see Saito's paper for a proof). Thus, it is an equivalence in cyclic homology
and Hochschild homology. Thus, the square on the left is homotopy
bi-cartesian. Next, note that%
\[
HC(\mathsf{Ind}_{\aleph_{0}}^{a}(\mathcal{C}))=0\qquad\text{and}\qquad
HC(\mathsf{Pro}_{\aleph_{0}}^{a}(\mathcal{C}))=0
\]
by the Eilenberg swindle, Lemma \ref{lem_EilenbergSwindle}. It is a rather
easy exercise to show that the Ind-category is closed under countable
coproducts, and the Pro-category under countable products. Thus, the homotopy
bi-cartesian square turns into the classical loop space square,%
\[
\bfig\Square(0,0)[HC(\mathcal{C})`\ast`\ast`HC(\mathsf{Tate}^{el}_{\aleph_{0}%
}\mathcal{C})\text{.};```]
\efig
\]
We obtain the equivalence $HC(\mathcal{C})\overset{\sim}{\rightarrow}\Omega
HC(\mathsf{Tate}_{\aleph_{0}}^{el}(\mathcal{C}))$, and analogously for
Hochschild homology\footnote{The same argument works for non-connective
$K$-theory, this is the original argument of Saito in \cite{MR3317759}, and
finally all of this generalizes to any localizing invariant in the sense of
Blumberg, Gepner, Tabuada \cite{MR3070515}.}. It remains to go to the
idempotent completion on the right; we copy the same argument as in the end of
the proof of Theorem \ref{Thm_KellersLocalizationTheorem}: For every exact
category $\mathcal{D}$ the idempotent completion $\mathcal{D}\hookrightarrow
\mathcal{D}^{ic}$ induces $D^{b}(\mathcal{D}^{ic})\hookrightarrow
D^{b}(\mathcal{D})^{ic}$, being co-final and an equivalence up to factors.
Thus, applied to $\mathcal{D}:=\mathsf{Tate}_{\aleph_{0}}^{el}(\mathcal{C})$,
we get $HC(\mathsf{Tate}_{\aleph_{0}}^{el}(\mathcal{C}))\cong HC(\mathsf{Tate}%
_{\aleph_{0}}(\mathcal{C}))$.
\end{proof}

Let us make the connection between this and the work of Fialowski and\ Iohara.

\begin{corollary}
\label{cor_1}Let $k$ be a field and $R$ a unital associative $k$-algebra. Then
there are canonical isomorphisms

\begin{enumerate}
\item $HC_{\bullet}(E(R))\cong HC_{\bullet-1}(R)$,

\item $HH_{\bullet}(E(R))\cong HH_{\bullet-1}(R)$.
\end{enumerate}
\end{corollary}

We observe that we have proven the analogue of the delooping result of
Fialowski and\ Iohara for $E$ instead of $J$.

\begin{proof}
To this end, we use the category-level delooping theorem in the special case
of rings, i.e. $\mathcal{C}=\mathsf{P}_{f}(R)$. Then%
\[
HC_{\bullet}(E(R))\cong HC_{\bullet}(\mathsf{Tate}_{\aleph_{0}}(\mathsf{P}%
_{f}(R)))\cong HC_{\bullet-1}(\mathsf{P}_{f}(R)))\cong HC_{\bullet
-1}(R)\text{.}%
\]
Here the first isomorphism stems from Lemma \ref{lemma_ModuleCat}, the second
from the delooping property of Tate categories, Theorem \ref{thm_deloop}, and
the last from the agreement of cyclic homology of exact categories of modules
with the classical definition for $k$-algebras. The same computation holds for
Hochschild homology.
\end{proof}

Let us stress how different this proof is from the one given in \cite{fioh}.
While loc. cit. uses the degeneration of a Hochschild--Serre type spectral
sequence, based on some explicit computations of terms, the main technical
point in our approach is Keller's Localization Theorem, along with a few
category-theoretic properties. In an at best vaguely precise comparison, the
r\^{o}le of the homological vanishing statements in \cite{fioh} appears to be
replaced by the\ Eilenberg swindle.

Finally, and for the first time in this note, let us restrict the
characteristic of the base field to zero. Then the Loday--Quillen--Tsygan
(LQT) theorem applies, telling us that for any unital associative $k$-algebra
$R$, we have%
\[
\operatorname*{Prim}H_{\bullet}(\mathfrak{gl}_{\infty}(R))\cong HC_{\bullet
-1}(R)\text{.}%
\]

\begin{theorem}
\label{thm_body_1}There is a canonical isomorphism of the primitive parts
$\operatorname*{Prim}H_{\bullet}(\mathfrak{gl}(\infty,R),k)\cong%
\operatorname*{Prim}H_{\bullet}(\mathfrak{gl}_{\infty}(R),k)[-1]$.
\end{theorem}

\begin{proof}
Write $\mathbb{M}_{n}(R)$ to denote the $(n\times n)$-matrix algebra over an
associative algebra $R$. Using Loday--Quillen--Tsygan twice, we obtain the
isomorphisms%
\[
\operatorname*{Prim}H_{\bullet}(\mathfrak{gl}_{\infty}(E(R)))\underset
{\text{(LQT)}}{\cong}HC_{\bullet-1}(E(R))\cong HC_{\bullet-2}(R)\underset
{\text{(LQT)}}{\cong}\operatorname*{Prim}H_{\bullet}(\mathfrak{gl}_{\infty
}(R))[-1]\text{,}%
\]
where the middle isomorphism is Corollary \ref{cor_1}. Moreover, if $Lie$
denotes the functor sending an associative $k$-algebra to its underlying Lie
algebra, then%
\begin{align}
\mathfrak{gl}_{n}(E(R))  & =Lie\left(  \mathbb{M}_{n}(E(R))\right)
\label{lab14}\\
& =Lie\left(  \mathbb{M}_{n}(\operatorname*{End}\nolimits_{\mathsf{Tate}%
_{\aleph_{0}}(\mathsf{P}_{f}(R))}(``R((t))\text{\textquotedblright}))\right)
\nonumber\\
& =Lie(\operatorname*{End}\nolimits_{\mathsf{Tate}_{\aleph_{0}}(\mathsf{P}%
_{f}(R))}(``R((t))\text{\textquotedblright}^{\oplus n}))\text{.}\nonumber
\end{align}
However, when writing $``R((t))\text{\textquotedblright}:=\coprod_{n\in\mathbb{N}%
}R\oplus\prod_{n\in\mathbb{N}}R$ as in Equation \ref{lab10}, we can write down
a system of isomorphisms%
\[
``R((t))\text{\textquotedblright}\cong``R((t))\text{\textquotedblright}^{\oplus2}%
\cong``R((t))\text{\textquotedblright}^{\oplus3}\cong\cdots\text{,}%
\]
induced from maps between the countable index sets of the product and
coproduct, as in \cite{fioh}, \S 1.3. (That is, essentially we use that the
disjoint union of $n\geq1$ copies of $\mathbb{N}$ is in bijection to
$\mathbb{N}$ itself). Being isomorphic as objects in $\mathsf{Tate}%
_{\aleph_{0}}(\mathsf{P}_{f}(R))$, their endomorphism algebras are also
isomorphic, and thus their Lie algebras. Hence, Equation \ref{lab14} yields
that the transition maps in
\[
\cdots\longrightarrow H_{\bullet}(\mathfrak{gl}_{n}(E(R)))\longrightarrow
H_{\bullet}(\mathfrak{gl}_{n+1}(E(R)))\longrightarrow\cdots
\]
are all isomorphisms, and thus $H_{\bullet}(\mathfrak{gl}_{1}(E(R)))\cong
H_{\bullet}(\mathfrak{gl}_{\infty}(E(R)))$ is an isomorphism. But clearly
$\mathfrak{gl}_{1}(E(R))$ is just a different way to denote the Lie algebra of
$E(R)$, so this is nothing but $\mathfrak{gl}_{\operatorname*{top}}(\infty
,R)$. This finishes the proof.
\end{proof}

Note that this statement is precisely the same as was proven by Fialowski
and\ Iohara for $\mathfrak{gl}(\infty,R)$ in their paper. In particular,
combining both results, we obtain that the homology of the plain uncompleted
Lie algebra $\mathfrak{gl}(\infty,R)$ acting on $R[t,t^{-1}]$, and the
topologically completed variant $\mathfrak{gl}_{\operatorname*{top}}%
(\infty,R)$, acting on $R((t))$, does not differ at all:

\begin{corollary}
$H_{\bullet}(\mathfrak{gl}(\infty,R),k)\cong H_{\bullet}(\mathfrak{gl}%
_{\operatorname*{top}}(\infty,R),k)$.
\end{corollary}

The delooping property for Hochschild homology would allow us to prove the
analogous statements for Leibniz homology; we leave this to the reader.
Finally, the same methods work for group homology:

\begin{theorem}
\label{thm_body_2}Let $R$ be an associative unital $k$-algebra. There is a
canonical isomorphism of primitive parts $\operatorname*{Prim}H_{n}%
(\operatorname{GL}_{\infty}(E(R)),\mathbb{Q})\cong\operatorname*{Prim}%
H_{n-1}(\operatorname{GL}_{\infty}(R),\mathbb{Q})$ for all $n\geq2$.
\end{theorem}

\begin{proof}
The proof follows exactly the same pattern as the one of Theorem
\ref{thm_body_1}; we only need to employ the picture (due to Feigin and Tsygan
\cite{MR923136}) to view cyclic homology as the additive analogue of
$K$-theory. Then take the arguments for cyclic homology and adapt them. In
detail: The Loday--Quillen--Tsygan theorem is modelled after the following
fact, essentially belonging to the realm of rational homotopy theory:%
\[
K_{n}(R)\otimes_{\mathbb{Z}}\mathbb{Q}\cong\operatorname*{Prim}H_{n}%
(\operatorname{GL}_{\infty}(R),\mathbb{Q})\text{,}%
\]
where the latter denotes the primitive part of the group homology of
$\operatorname{GL}(R)$, equipped with the discrete topology. Moreover,
$K_{\bullet}(-)$ denotes the usual Quillen (that is: connective) $K$-theory.
Now replace Theorem \ref{thm_deloop} by Saito's original theorem for
(non-connective) $K$-theory \cite{MR3317759}. We get $K_{n+1}(\mathsf{Tate}%
_{\aleph_{0}}(\mathsf{P}_{f}(R)))\cong K_{n}(\mathsf{P}_{f}(R))$ for all
$n\geq1$; because for $n\geq1$ the connective and non-connective $K$-theories
agree. Thus,%
\begin{align*}
\operatorname*{Prim}H_{n+1}(\operatorname{GL}_{\infty}(E(R)),\mathbb{Q})  &
\cong K_{n+1}(E(R))\otimes\mathbb{Q}\cong K_{n+1}(\mathsf{Tate}_{\aleph_{0}%
}(\mathsf{P}_{f}(R)))\otimes\mathbb{Q}\\
& \cong K_{n}(\mathsf{P}_{f}(R))\otimes\mathbb{Q\cong}\operatorname*{Prim}%
H_{n}(\operatorname{GL}_{\infty}(R),\mathbb{Q})\text{.}%
\end{align*}
This finishes the proof.
\end{proof}

\begin{acknowledgement}
I heartily thank M. Groechenig and J. Wolfson. I am most grateful to K. Iohara
for his interest and correspondence.
\end{acknowledgement}

\bibliographystyle{amsalpha}
\bibliography{acompat,ollinewbib}

\end{document}